\documentclass[10pt]{amsart}
\usepackage[utf8]{inputenc}
\usepackage{amsfonts}
\usepackage{hyperref}
\usepackage{amsmath}
\usepackage{amsthm}
\usepackage{mathrsfs}
\usepackage{tikz}
\usepackage{amssymb}
\usepackage{faktor}
\usepackage[all]{xy}

\newtheorem{thm}{Theorem}[section]
\newtheorem{cor}[thm]{Corollary}
\newtheorem{prop}[thm]{Proposition}
\newtheorem{lem}[thm]{Lemma}

\theoremstyle{definition}
\newtheorem{defn}[thm]{Definition}
\newtheorem{ex}[thm]{Example}
\newtheorem{exs}[thm]{Examples}

\theoremstyle{remark}
\newtheorem{rem}[thm]{Remark}

\newcommand{\R}{\mathbb{R}}
\newcommand{\CC}{\mathbb C}

\newcommand{\Z}{\mathbb Z}
\newcommand{\cS}{\mathcal{S}}
\newcommand{\CCC}{\mathcal{C}}
\newcommand{\A}{\mathcal{A}}
\newcommand{\AAA}{\mathfrak{A}}
\newcommand{\F}{\mathcal{F}}
\newcommand{\pr}{\operatorname{pr}}
\newcommand{\T}{\mathcal{T}}
\newcommand{\ev}{\operatorname{ev}}

\title[Toeplitz operators and symplectic groupoids]{Symplectic algebroids, groupoid Toeplitz operators and deformation quantization}
\author{Clement Cren, Jean-Marie Lescure, Omar Mohsen}
\begin{document}

\begin{abstract}
    We use Toeplitz operators to define a star-product on Poisson manifolds whose Poisson structure is induced by a symplectic Lie algebroid.
    The Toeplitz operators we consider are defined on groupoids whose algebroid can be endowed with a Heisenberg group structure on the fibers. This generalizes an approach due to Guillemin and Melrose in the symplectic case.
\end{abstract}

\maketitle

\section{Introduction}
Let  \((M,\{\cdot,\cdot\})\) be a Poisson manifold. A deformation quantization of the Poisson structure is an algebra structure on the vector space formal power series
\[\ast \colon \CCC^{\infty}(M)[[t]] \otimes \CCC^{\infty}(M)[[t]] \to \CCC^{\infty}(M)[[t]]\]
such that for all \(f,g \in \CCC^{\infty}(M)\)
\begin{align*}
     & f\ast g = fg    +O(t)    ,                       \\
     & f\ast g - g\ast f = t \lbrace f,g\rbrace + O(t).
\end{align*}
Kontsevich \cite{Kontsevich2003} proved algebraically that any Poisson manifold admits a deformation quantization. %He constructed an isomorphism between the formal deformations of Poisson structures (which a Poisson structure canonically induces) and the deformation quantizations of \(C^{\infty}(M)\) as above (up to equivalence on both sides). 
In the special case of symplectic manifolds, several constructions of deformation quantization exist. We are interested here in the construction of Guillemin for prequantizable symplectic manifolds \cite{Guillemin1995}, later generalized by Melrose to arbitrary symplectic manifolds \cite{Melrose2004}.
This construction uses the class of pseudodifferential operators on contact manifolds introduced in \cite{BoutetdeMonvelGuillemin1981}, also called Toeplitz operators. These operators are defined by compression of the pseudodifferential operators, using a certain projection, the so-called Szegö projection.
The main idea of Guillemin and Melrose is that the symbol algebra of Toeplitz operators is isomorphic to $C^\infty(M)[[t]]$. The star product then comes from composition of pseudo-differential operators.
%This allows to construct a star-product on the underlying symplectic manifold by considering the algebra of Toeplitz operators on a prequantization of it. The same construction can be performed on any contact manifold, and it yields a star-product on the leaf space of the Reeb foliation. In particular, if the Reeb foliation is simple, we get a star-product on the quotient manifold.
%This compression makes their symbolic calculus easier and the algebra of total symbol can be identified with formal power series in the algebra of smooth functions on the underlying contact manifold.

In this paper, we adapt this method to Poisson structures induced by a symplectic algebroid introduced by Nest and Tsygan \cite{NestTsygan2001}, recovering their construction of the star-product in an analytic way.

Our strategy is the following: using the symplectic form on the algebroid \(\A_s \to M\), we can construct another algebroid \(\A_c = \A_s \oplus \R \to M\) which is the contact analog of $\A_s$. The algebroid $\A_c$ is filtered and the corresponding grading is a bundle of Heisenberg Lie algebras (this is where the symplectic structure is crucial).
If the algebroid $\A_c$ is integrable, and \(G_c\) is a groupoid integrating \(\A_c\), we can construct pseudodifferential operators on \(G_c\) in the Heisenberg calculus. This calculus allows us to define a Szegö projector (at least formally, the operator itself does not need to be a projector).
This operator yields an algebra of Toeplitz operators, for which the symbol algebra is $C^\infty(M)[[t]]$. We then show that the composition of pseudodifferential operators gives a star product.

The algebroid \(\A_c\) does not need to be integrable in general. However, the previous construction can be done in any neighborhood of the unit and can thus be performed on a local groupoid, which always exists.

The paper is organized as follows. In Section \ref{Section: Symplectic algebroids} we introduce the geometric context, symplectic algebroids, and give examples of Poisson structures arising from them.
We also define the contact analog of these algebroids by means of a canonical central extension, which corresponds to a prequantization of the symplectic Lie algebroid. In Section \ref{Section: Groupoid integration} we explain the structure of the (local) groupoids integrating these algebroids. They happen to have a right-invariant symplectic (or contact) structure on their range fibers. Groupoids integrating the symplectic algebroid can thus be seen as a family of symplectic manifolds desingularising the initial Poisson structure.
Groupoids integrating the contact algebroid extension can then be seen as a family of contact manifolds prequantizing the previous family of symplectic ones. In Section \ref{Section: Groupoid Heisenberg calculus} we review some features of the Heisenberg calculus on groupoids whose algebroid is filtered by a corank one subbundle.
This calculus and its symbolic properties allow us in Section \ref{Section: Toeplitz operators} to construct a (formal) Szegö projector and study the properties of the corresponding Toeplitz operators.
In Section \ref{Section: Star product} we use these operators to construct a star product on Poisson manifolds whose Poisson structure is induced by a symplectic algebroid.

\section{Symplectic algebroids}\label{Section: Symplectic algebroids}
We refer the reader to \cite{Mackenzie2005} for the general theory of Lie algebroids and Lie groupoids.
\begin{defn}
    A Lie algebroid \(\A_s \to M\) is called symplectic if there exists a 2-form \(\omega \in \Gamma(M,\Lambda^2 \A^*)\) which is closed for the Lie algebroid differential and non-degenerate.
\end{defn}
Recall that $\omega$ being closed for the Lie algebroid differential means that for every \(X,Y,Z \in \Gamma(M,\A_s)\), we have
\begin{eqnarray*}
    &\rho(Z)(\omega(X,Y) )+\rho(Y)(\omega(Z,X) )+\rho(X)(\omega(Y,Z) )\\=&\omega([X,Y],Z) + \omega([Y,Z],X) + \omega([Z,X],Y),
\end{eqnarray*}
where \(\rho \colon \A_s \to TM\) is the anchor map.
A symplectic algebroid induces a Poisson structure on $M$ as follows: Let \(\omega^{\#} \colon \A_s^* \to \A_s\) be the inverse of the contraction map along \(\omega\).
We define a map \(\pi^{\#} \colon T^*M \to TM\) by the following diagram,

\[\xymatrix{\A_s \ar[r]^{\rho} & TM \\
    \A_s^* \ar[u]^{\omega^{\#}} & T^*M \ar[l]^{\rho^*} \ar[u]_{\pi^{\#}}}\]

The map \(\pi^{\#}\) is the contraction by a 2-tensor \(\pi \in \Gamma(M,\Lambda^2TM)\).

\begin{prop}\label{Proposition: Poisson structure from symplectic algebroid}
    The tensor \(\pi\) is a Poisson tensor. The induced symplectic leaves are contained in the leaves of the foliation induced by the Lie algebroid.
\end{prop}
\begin{proof}

    We need to check that \([\pi,\pi] = 0\). Let \(\pi_0 \in \Gamma(M,\Lambda^2\A_s)\) be the two tensor such that the contraction by \(\pi_0\) is the map \(\omega^{\#}\). This tensor satisfies \([\pi_0,\pi_0] = 0\). The map \(\rho \colon \A \to TM\) preserves the Lie algebra structure on sections thus so does its extension \(\Gamma(M, \Lambda^{\bullet}\A_s) \to \Gamma(M,\Lambda^{\bullet}TM)\). Since \(\pi = \rho(\pi_0)\) we get \([\pi,\pi] = \rho([\pi_0,\pi_0]) = 0\).

    The tangent space of the symplectic leaves are given by \(\pi^{\#}(T^*M)\). The tangent space of the Lie algebroid foliation is given by \(\rho(\A_s)\). From the diagram we have \(\pi^{\#}(T^*M) \subset \rho(\A_s)\), hence the inclusion of leaves.
\end{proof}

Let \((M,\pi)\) be a Poisson manifold. We get a Lie algebroid \(T^*M \to M\) for which the Poisson tensor becomes a 2-form \(\pi \in \Lambda^2(T^*M)^*\). It is however non-degenerate if and only if \(M\) is symplectic.

\begin{exs}
    \begin{enumerate}
        \item         The tangent Lie algebroid \(TM\) of a manifold \(M\) is symplectic if and only if \(M\) is a symplectic manifold.
        \item     A Lie algebra \(\mathfrak{g}\) is symplectic if and only if it is a quasi-Frobenius Lie algebra. This is a particular case of triangular Lie bialgebra where the \(r\)-matrix is invertible (its inverse will be the symplectic structure). These structures are tied to Yang-Baxter equations and quantum deformations, see e.g. \cite{ChariPressley1994}.
        \item     A \(E\)-symplectic manifold is a manifold \(M\) with an involutive sub-module of vector field \(E \subset \mathfrak{X}(M)\) which is locally free, and such that the corresponding Lie algebroid \(\!^ETM\) (whose sections are elements of \(E\)) is symplectic. Such manifolds have received a lot of attention recently (most notably \(b\)-symplectic manifolds) \cite{MirandaScott2021,GuilleminMirandaPires2014,MirandaOms2023}.
    \end{enumerate}
\end{exs}

We specialize the previous class of examples into several ones.

\begin{ex}[b-symplectic]
    Let \(M\) be a smooth manifold, \(Z = f^{-1}(0)\subset M\) a hypersurface. A b-symplectic structure is a Poisson structure which in coordinates \((f,x_2,\cdots,x_{2n})\) takes the form:
    \[\pi = f\frac{\partial}{\partial f}\wedge \frac{\partial}{\partial x_{n+1}} + \sum_{j = 2}^{n}\frac{\partial}{\partial x_{j}}\wedge \frac{\partial}{\partial x_{n+j}} .\]
    The symplectic leaves of this structure are of two kind: connected components of \(M\setminus Z\) and the hypersurfaces in \(Z\) given by level set of \(x_{n+1}\) in the coordinates above.
    This Poisson structure is induced by a symplectic structure on the \(b\)-algebroid, whose sections are vector fields tangent to the hypersurface \(Z\). This algebroid is locally generated by the following vector fields \(f\frac{\partial}{\partial f}, \frac{\partial}{\partial x_{j}}, j\geq 2\). The symplectic structure takes the form:
    \[\omega = \frac{\mathrm{d}f}{f}\wedge \mathrm{d}x_{n+1} + \sum_{j = 2}^n \mathrm{d}x_j\wedge \mathrm{d} x_{n+j},\]
    with \(\frac{\mathrm{d}f}{f}\) being a well defined section of the dual of the \(b\)-algebroid.
\end{ex}

\begin{ex}[0-symplectic]
    We keep the same notations as previously, this time the Poisson structure has the form
    \[\pi = f^2\frac{\partial}{\partial f}\wedge \frac{\partial}{\partial x_{n+1}} + f^2\sum_{j = 2}^{n}\frac{\partial}{\partial x_{j}}\wedge \frac{\partial}{\partial x_{n+j}} .\]
    The leaves are given by connected components of \(M\setminus Z\) and points in \(Z\).
    This structure is induced by a symplectic structure on the \(0\)-algebroid, whose sections are vector fields vanishing on \(Z\). It is generated by vector fields of the form \(fX, X \in \mathfrak{X}(M)\).
    The symplectic structure takes the form:
    \[\omega = \frac{1}{f^2}\left( \mathrm{d}f \wedge \mathrm{d}x_{n+1}+ \sum_{j = 2}^n \mathrm{d}x_j\wedge \mathrm{d} x_{n+j}\right).\]
\end{ex}

These examples showcase a symplectic structure on an open dense subset, which degenerates on its boundary in various ways, leading to different Poisson structures. We can also have foliated behaviors on the interior.

\begin{ex}[Regular symplectic foliations]
    If the symplectic foliation is regular, the algebroid is given by the tangent space to the leaves. The Poisson structure then induces a symplectic structure on the algebroid in a tautological way. Toeplitz operators have been considered in that context (when the foliation is a fibration) in \cite{CrenRezaei2026}.
\end{ex}

These examples can also be combined, for instance one could consider a foliated manifold, whose leaves are b-symplectic. The b-tangent bundles of the leaves would then induce a Lie algebroid on the manifold which is also symplectic...

\begin{rem}
    On a \(E\)-symplectic manifold, the anchor map is injective over an open dense subset \(M_{reg}\). Over this subset \(M_{reg}\) the image of \(\pi^{\#}\) then coincides with the one of \(\rho\) and the symplectic leaves are exactly the ones of the foliation induced by the Lie algebroid (this foliation also happens to be regular over \(M_{reg}\)).

    On the whole manifold however, the two foliations may differ. In the \(b\)-symplectic case for instance, both foliations will have an open dense leaf (which is the regular part as above), but they differ on the hypersurface. The algebroid foliation only has one other leaf which is the whole hypersurface. On the other hand, the symplectic leaves form a foliation of codimension 1 of the boundary.

    The class of Poisson structures induced by a symplectic algebroid is also not well established in the literature. There are however examples of Poisson structures that are induced by symplectic algebroids but that are not \(E\)-symplectic.

    An example of Poisson structure that cannot be \(E\)-symplectic is the following. Consider the open unit ball \(B \subset\R^{2n}\). Let \(\pi_{st}\) be the Poisson structure corresponding to the standard symplectic structure on \(\R^{2n}\). Let \(\varphi \colon \R^{2n} \to \R_{\geq 0}\) be a function which is positive on \(B\) and vanishes on its complement. The tensor \(\pi := \varphi\pi_{st}\) defines a new Poisson structure on \(\R^{2n}\). Its symplectic leaves consist of the open leaf \(B\) and all the singletons \(\{x\}\) with \(x \in \R^{2n}\setminus B\). Consider the algebroid \(\A_s \to \R^{2n}\) whose sections are generated by elements \(X_1,\cdots,X_{2n}\), and whose anchor map is \(X_j \mapsto \varphi\frac{\partial}{\partial x_j}\). It is a symplectic Lie algebroid for the symplectic form \(\omega := \sum_{j = 1}^n X_j^*\wedge X_{n+j}^*\). The Poisson structure induced by \(\A_s\) is clearly \(\pi\). It can however not be induced by an \(E\)-symplectic structure. Indeed, the rank of the Poisson structure would then have to be maximal on an open dense subset (where the algebroid is injective) which is not the case here.
\end{rem}

Let \((\A_s,\omega)\) be a symplectic algebroid. The 2-form \(\omega\) is an algebroid 2-cocycle, we can thus construct the following central extension:
\[\xymatrix{0 \ar[r] & \R \ar[r] & \A_c \ar[r] & \A_s \ar[r] & 0.}\]
As a vector bundle, \(\A_c := \A_s \oplus \R\) and the Lie bracket has the following form:
\begin{equation}\label{Equation: Central extension Lie algebroid}
    \forall X,Y \in \Gamma(M,\A_s), \forall f,g\in \CCC^{\infty}(M), [(X,f),(Y,g)] = ([X,Y],\omega(X,Y)).
\end{equation}

\section{Integration of symplectic algebroids and their extensions}\label{Section: Groupoid integration}

We keep the notations of the previous section. We denote by \((\A_s, \omega)\) a symplectic Lie algebroid and by \(\A_c\) the corresponding Lie algebroid extension by the 2-cocycle \(\omega\).

Consider the one form \(\theta \in \Omega^1(\A_c)\) given by
\[\forall (X,t) \in A_c, \quad \theta(X,t) = - t.\]
This one form has kernel \(\A_s\) and using Equation \ref{Equation: Central extension Lie algebroid}, its differential is the pullback of \(\omega\) to \(\A_c\). This form should be seen as a contact form on the algebroid \(\A_c\), and the central extension as a prequantization of the symplectic structure \((\A_s, \omega)\) to the contact structure \((\A_c,\theta)\). We make sense of these intuition by integrating the algebroids \(\A_s\) and \(\A_c\).

Assume for now that the algebroids \(\A_s\) and \(\A_c\) are integrable. Denote by \(G_s\rightrightarrows M\) and \(G_c\rightrightarrows M\) the \(s\)-connected, \(s\)-simply connected Lie groupoids integrating \(\A_s\) and \(\A_c\) respectively. The forms \(\omega\) and \(\theta\) induce right-invariant forms on the groupoids \(G_s\) and  \(G_c\), which we will denote by \(\Omega\) and \(\Theta\) respectively.

\begin{prop}\label{Proposition: Groupoid structure fibers}
    For each \(x \in M\), the range-fibers \((G_s)^x\) and \((G_c)^x\) are symplectic and contact manifolds for the restrictions of the form \(\Omega\) and \(\Theta\) respectively.
\end{prop}

\begin{proof}
    Let \(\gamma \in G_s\) with \(r(\gamma) = x\). We have \(T_{\gamma}(G_s)^x = \ker(\mathrm{d}r)_{\gamma}\), hence \(T_{\gamma}(G_s)_x \cong R_{\gamma}^*(\A_s)_x\). This linear isomorphism is a symplectomorphism for the restriction of \(\Omega_{\gamma}\) and \(\omega_x\) respectively. In particular the restriction of \(\Omega\) to the range fibers is non-degenerate. Now the restriction of \(\mathrm{d}\Omega\) to a range fiber is the right-invariant form corresponding to \(\mathrm{d}\omega = 0\) on the algebroid. Therefore \((G_s)^x\) is a symplectic manifold for the restriction of \(\Omega\). A similar reasoning holds for the contact case.
\end{proof}

The structure of the groupoid \(G_s\) was also identified in \cite{Pham2021}, with a more systematic approach on groupoids integrating symplectic algebroids. Such groupoids are then called \(r\)-symplectic Lie groupoids (\(r\) being the range map). In the same fashion, groupoid integrating the central extension of a symplectic algebroid could be called \(r\)-contact Lie groupoids.

\begin{rem}
    Let \(H \subset \ker(\mathrm{d}r) \subset TG_c\) be the right-invariant subbundle corresponding to \(\A_s \subset \A_c\). Let \(R\) be the right invariant vector field on \(G_c\) corresponding to the section
    \[x \mapsto (x,(0_{\A_s}, 1))\]
    of the algebroid \(\A_c\). The form \(\Theta\) vanishes on \(H\) and \(\Theta(R) = 1\). Therefore \(H\) restricts on each fiber to the contact distribution and \(R\) to the Reeb vector field.
\end{rem}

\begin{prop}
    The groupoid \(G_s \rightrightarrows M\) is a Poisson groupoid. In particular, the source map is a Poisson map and the target is anti-Poisson.
\end{prop}

\begin{proof}
    Since \(\omega\) is non-degenerate, we can define \(\pi \in \Gamma(\Lambda^2\A)\) so that contraction by \(\pi\) is the inverse of contraction by \(\omega\). The identity \(\mathrm{d}\omega = 0\) then translates to \([\pi,\pi] = 0\). The pair \((\A_s,\A_s^*)\) is then a Lie bialgebroid and the groupoid \(G_s\) is a Poisson groupoid \cite{MackenzieXu1994,LiuXu1996}.
\end{proof}

This proposition is the reason why we put the symplectic (and contact) structure on the range fibers. Indeed, for \(x \in M\), we get a Poisson map \(s \colon (G_s)^x \to G_s \cdot x\) where \(G_s \cdot x\) is the orbit of the point \(x\) under the groupoid action. We identified the orbit to a Poisson manifold in Proposition \ref{Proposition: Poisson structure from symplectic algebroid} and the range fiber to a symplectic manifold in the previous proposition. The groupoid \(G_s\) can thus be seen as a family of symplectic realizations for the orbits it induces on the base.

In the next example, we relate the roles of the groupoids \(G_s\) and \(G_c\) to the objects used when prequantizing symplectic manifolds.

\begin{ex}
    Let \((M,\omega)\) be a (connected, simply connected) symplectic manifold. The algebroid \(TM\) is then symplectic. It is integrated by the pair groupoid \(M \times M \rightrightarrows M\) and we have \(\Omega = \pr_1^*\omega - \pr_2^*\omega\). Assume \(M\) is prequantizable, this means there exists a principal \(\mathbb{S}^1\)-bundle \(p \colon X \to M\) with a connection \(\theta\) satisfying \(\mathrm{d}\theta = p^*\omega\). The manifold \(X\) then has the structure of a contact manifold with one form \(\theta\). We can form the gauge groupoid \(G = (X \times X)/\mathbb{S}^1 \rightrightarrows M\). This groupoid sits in the central extension
    \[ M \times \mathbb{S}^1 \hookrightarrow G \twoheadrightarrow M \times M.\]
    At the level of Lie algebroids, this extension becomes the central extension of the symplectic algebroids \(TM\).
\end{ex}

\begin{rem}
    Notice in the previous example that the gauge groupoid is not source simply connected so it would not be the groupoid \(G_c\) but a quotient of it. In general one can define the so called periods groupoid \(\mathrm{Per}(\omega) \subset M \times \R\). The groupoid generalizing the gauge groupoid of the previous example is the quotient of \(G_c\) by the image of the periods under the map \(M\times \R \to G_c\). If we write \(\widetilde{G}_c\) for this new groupoid, we obtain the central extension of Lie groupoids
    \[ M \times \mathbb{S}^1 \hookrightarrow \widetilde{G}_c \twoheadrightarrow G_s,\]
    instead of the one of \(G_c\):
    \[ M \times \R \hookrightarrow G_c \twoheadrightarrow G_s.\]
    In light of Proposition \ref{Proposition: Groupoid structure fibers}, the map \(\widetilde{G}_c \twoheadrightarrow G_s\) then restricts on each range fiber to the \(\mathbb{S}^1\)-principal bundle map prequantizing the symplectic manifold \((G_s)^x\). This philosophy of applying the quantization procedure to a groupoid originates in \cite{WeinsteinXu1991} for symplectic groupoids. Here the symplectic structure is only on the range fibers but the idea is similar.

    In this paper, we are only interested in the neighborhood of the units in \(G_c\) so considering \(G_c\) or \(\widetilde{G}_c\) doesn't matter. Note that in general, algebroids need not being integrable \cite{CrainicFernandes2003}. We can however integrate them by local groupoids. These local Lie groupoids will inherit symplectic symplectic and contact structure as above. The local Lie groupoid integrating an algebroid is unique up to equivalence. Therefore with this local point of view, we do not see the difference between \(G_c\) and \(\widetilde{G}_c\) anyway.
\end{rem}

Now an algebroid is not always integrable. Even if \(G_s\) and \(G_c\) exist as topological groupoids, they may fail to have a compatible smooth structure. We can however always assume that there are local Lie groupoids.

\begin{defn}
    A local Lie groupoid \(G\) over a manifold \(M\) is a manifold with:
    \begin{itemize}
        \item source and target maps \(s,r \colon G \to M\) which are submersions
        \item a unit map \(u \colon M \to G\) which is an immersion
        \item a multiplication map \(m \colon \mathcal{U} \to G\) where \(\mathcal{U} \subset G^{(2)}\) is an open neighborhood of \(G \ _r\!\times_s M \cup M \ _r\!\times_s G\)
        \item an inversion map \(i \colon \mathcal{V} \to \mathcal{V}\)
    \end{itemize}
    These maps have to satisfy the usual relations of groupoid structure maps, whenever it makes sense.
\end{defn}

\begin{rem}
    By `whenever it makes sense', we mean for instance that the associativity relation:
    \[m(m(g,h),k) = m(g,m(h,k))\]
    has to hold for \((g,h,k)\in G^{(3)}\) such that also \((g,h),(h,k),(m(g,h),k)\) and \((g,m(h,k))\) all belong to \(\mathcal{U}\).
\end{rem}
%MAYBE PRADINES INSTEAD FOR LOCAL INTEGRATION ?
Every local Lie groupoid has a Lie algebroid the same way a groupoid does. It is proved in \cite{CrainicFernandes2003} that every algebroid integrates into a local Lie groupoid. Moreover, two such local Lie groupoids will have neighborhoods of the unit which are isomorphic (as local Lie groupoids). Finally, as on a Lie groupoid, every algebroid differential form (or section of the algebroid) induces a unique differential form (or vector field) on a neighborhood of the units which is right invariant.

\section{Heisenberg calculus on a groupoid}\label{Section: Groupoid Heisenberg calculus}

In this section we expose the features of the Heisenberg calculus on a groupoid. Given a Lie groupoid \(G\) with algebroid \(\A\). One can construct a pseudodifferential calculus for which the kernels of operators will be distributions on \(G\), conormal to the units. The symbols are then functions on \(\A^*\). The differential operators on that calculus are constructed from the sections of the Lie algebroid, which act as order 1 operators.
%Consider now a groupoid \(G\) whose algebroid \(\A\) is filtered by subbundles:
%\[0 = \A_0 \subset \A_1 \subset \cdots \subset \A_N = \A,\]
%with the condition on the Lie brackets
%\[\forall i,j\geq 1, [\Gamma(\A_i),\Gamma(\A_j)] \subset \Gamma(\A_{i+j}).\]
%We can modify the order of differential operators by considering sections of \(\A_k\) as operators of order \(k\). One can then adapt the construction of the pseudodifferential calculus to these operators. The symbols are now distributions on a family of graded nilpotent groups obtained from the filtration on the algebroid. In the case that will be of interest for us, they will be Heisenberg groups.
Consider now a groupoid \(G\) whose algebroid \(\A\) is filtered by a co-rank 1 subbundle \(\A_1 \subset \A\). We can consider a new order on differential operators by considering sections of \(\A\) that are not section of \(\A_1\) as having order 2 instead of 1. The construction of pseudodifferential operators is then modified to take into account this new order. The symbols are then operator valued functions on the unitary dual of a family of nilpotent group. If \(\A\) is the central extension of a symplectic algebroid \(\A_1\) then these groups will be Heisenberg groups. Historically, this calculus has been invented to study the hypoelliptic properties of sub-laplacians in a sub-riemannian setting (e.g. contact manifolds). We are here interested in the symbolic properties. Since the principal symbols are families of operators on Hilbert spaces, they have non-trivial projectors. This allows to construct the Szegö projector: a projector \(S\) such that both its kernel and image have infinite dimension. On a general groupoid it is not clear that such a projector would exist, but for our purpose we only need such an operator to be a projector modulo smoothing operators.

Let \(\A \to M\) be a Lie algebroid and \(\A_1\subset \A\) a co-rank 1 subbundle. We will assume that the vector bundle \(\A\) decomposes as \(\A = \A_1 \oplus \R\). Define
\begin{align*}
    \omega \colon \Gamma(\A_1) \wedge \Gamma(\A_1) & \to \CCC^{\infty}(M)               \\
    X\wedge Y                                      & \mapsto \pr_{\R}[(X,0),(Y,0)]_{\A}
\end{align*}
This bilinear form is tensorial, indeed with similar notations:
\[[fX,Y] = \rho(X)(f) Y + f[X,Y] \equiv f [X,Y] \mod \A_1.\]
Therefore we can consider \(\omega\) as a 2-form, \(\omega \colon \Lambda^2\A^1 \to \R\).

\begin{ex}
    If \(\A_1\) is a Lie algebroid with a 2-cocycle \(c \colon \Lambda^2\A_1 \to \R\), then consider its central extension \(\A = \A_1 \oplus \R\) with the Lie bracket given by Equation \ref{Equation: Central extension Lie algebroid}. The algebroid \(\A\) is filtered by \(\A_1\), and the 2-form \(\omega\) constructed above is equal to \(c\).
\end{ex}

Using the form \(\omega\), we can define a nilpotent group structure on the fibers of \(\A\). First the 2-form \(\omega\) defines a Lie algebra structure on the fiber \(\A_x\) by:
\[\forall x\in M,\forall (\xi_1,t_1),(\xi_2,t_2) \in \A_x, [(\xi_1,t_1),(\xi_2,t_2)]_{\A_x} := (0,\omega_x(\xi_1,\xi_2)). \]
This makes \(\A_x\) a graded nilpotent Lie algebra of step 2. Using Baker-Campbell-Hausdorff formula we can consider \(\A_x\) as the nilpotent group integrating it. This endows \(\A\) with a smooth family of Lie group structure on the fibers (i.e. a Lie groupoid with equal source and range). We denote \(\A\) endowed with this new structure of Lie groupoid by \(\mathfrak{A}\), we may also use this notation for the family of Lie algebras indifferently.

\begin{ex}
    Abelian Lie groups are examples of graded groups. In that case we can either consider them with every element of degree 1, or consider direct sums \(V_1 \oplus V_2\) of abelian groups where elements of \(V_1\) have degree 1 and those in \(V_2\) have degree 2.

    Another example is given by the Heisenberg groups. Their Lie algebra is generated by elements \(X_1,\cdots,X_n,Y_1,\cdots,Y_n,Z\) with the relations \([X_i,Y_i] = \delta_{i,j}Z\), \(Z\) is in the center and is the only of these elements of degree 2. We can describe it more geometrically by considering a symplectic vector space \((V,\omega)\). The form \(\omega\) is a 2-cocycle on the abelian group \(V\) and the corresponding central extension will be denoted \(\mathrm{Heis}(V,\omega)\). If we choose a Darboux basis \(X_1,\cdots,X_n,Y_1,\cdots,Y_n\) we recover the previous definition of the Heisenberg group with \(Z\) being the added variable of the central extension.
\end{ex}

The group structure at a point \(x\) depends on the rank of the 2-form \(\omega_x\).

\begin{lem}
    Let us decompose \(\A_x = V_x \oplus K_x \oplus \R\) such that \(V_x\) and \(K_x\) are orthogonal for \(\omega_x\), \(\omega_{x|K_x} = 0\) and \(\omega_x\) is non-degenerate on \(V_x\). Then we have:
    \[\AAA_x = \mathrm{Heis}(V_x,\omega_x) \oplus K_x,\]
    where \(K_x\) is seen as an abelian group in degree 1.
\end{lem}

\begin{proof}
    The decomposition of \(\A_x\) can be done by choosing an appropriate basis for the 2-form \(\omega_x\). The result then follows directly from the definition of the osculating group.
\end{proof}

Again the case of interest in this paper is the algebroids obtained by central extension of a symplectic algebroid. In that case we obtain the following.

\begin{cor}\label{Corollary: Osculating central extension}
    If \(\A\) is the central extension of a symplectic algebroid \((\A_1,\omega)\) then the osculating groupoid \(\AAA = \mathrm{Heis}(\A_1,\omega)\) is a bundle of Heisenberg groups.
\end{cor}

We now briefly describe how to construct Heisenberg pseudodifferential operators on such a groupoid, and how to obtain principal symbols. We mostly follow the approach of \cite{vanErpYuncken2019}. This uses the adiabatic groupoid of \(G\). As a groupoid is has the form
\[G^{ad} = G \times \R^* \sqcup \AAA \times \{0\} \rightrightarrows M \times \R.\]

The structure at time \(t \neq 0\) is the one of the groupoid \(G\) and at \(0\) of the osculating groupoid \(\AAA\). It has a natural smooth structure making it a Lie groupoid, which can be constructed with a variation of a deformation to the normal cone, see \cite{Mohsen2021}.

The adiabatic groupoid is endowed with the zoom action. For \(\lambda > 0\) we have
\begin{align*}
    \alpha_{\lambda} \colon G^{ad} & \mapsto G^{ad}                                \\
    (\gamma,t)                     & \mapsto (\gamma, \lambda^{-1}t) \quad t\neq 0 \\
    (x,\xi,0)                      & \mapsto (x,\delta_{\lambda}\xi,0).
\end{align*}
The map \(\delta_{\lambda} \colon \AAA \to \AAA\) is a groupoid homomorphism. On the Lie algebroid, its differential acts as \(\lambda \mathrm{Id}\) on \(\A_1\) and as \(\lambda^2\mathrm{Id}\) on \(\faktor{\A}{\A_1}\). The map \(\alpha_{\lambda}\) is a groupoid homomorphism and the map \(\lambda \to \alpha_{\lambda}\) is a group homomorphism from the positive real numbers to the groupoid automorphisms of \(G^{ad}\).

\begin{rem}
    The groupoid described above is a variation on the usual adiabatic groupoid, whose restriction over \(M \times 0\) would be the algebroid \(\mathcal{A}G\) (seen as a bundle of abelian groups). The groupoid described above is the Heisenberg/filtered version of the adiabatic groupoid. Since we will not use the unfiltered version, we will call adiabatic groupoid the one taking the filtration into account.
\end{rem}

We see pseudodifferential operators through their Schwartz kernel. This means that a pseudodifferential operator on a groupoid will be a particular kind of distribution on \(G\). We denote by \(\mathcal{E}'(G)\) the set of properly supported distributions on \(G\), a set \(X \subset G\) being proper if \(r_{|X}\) and \(s_{|X}\) are proper maps.

\begin{defn}
    A distribution \(P \in \mathcal{E}'(G)\) is a Heisenberg pseudodifferential operator of order \(n  \in \Z\) if there exists a distribution \(\mathbb{P} \in \mathcal{E}'(G^{ad})\) such that:
    \begin{itemize}
        \item \(\ev_{1*}\mathbb{P} = P\)
        \item \(\mathbb{P}\) is transverse to the range map,
        \item \(\mathbb{P}\) is quasi-homogeneous of order \(n\) for the zoom action
              \[\forall \lambda > 0, \alpha_{\lambda *}\mathbb{P} - \lambda^n \mathbb{P} \in \CCC^{\infty}(G^{ad}).\]
    \end{itemize}
    We denote by \(\Psi^n_H(G^{ad})\) the set of such distributions.
\end{defn}

Composition of operators is obtained by convolution of their Schwartz kernels. For pseudodifferential operators we obtain products:
\[\Psi^n_H(G) \times \Psi^m_H(G) \to \Psi^{n+m}_H(G), \quad n,m\in \mathbb{Z}.\]

Given an operator \(P\in \Psi^n_H(G)\) and two extensions \(\mathbb{P}, \mathbb{P}'\) as in the definition, then \(\ev_{0*}\mathbb{P} - \ev_{0*}\mathbb{P}' \in \CCC^{\infty}(\AAA)\). This allows to define a principal symbol
\[\sigma^n_H \colon \Psi^n_H(G) \to \Sigma^n(\AAA)\]
where \(\Sigma^n(\AAA)\) is the quotient of the space of order \(n\) quasi-homogeneous distributions on \(\AAA\) by the space of smooth functions. This space can be identified with a space of homogeneous distributions on the groupoid \(\AAA\) or, using a result of Taylor \cite{Taylor1984}, to smooth functions on \(\AAA^*\setminus M\) which are homogeneous of order \(n\) for the action \(\delta_{\lambda}^T, \lambda > 0\).

\begin{prop}[\cite{vanErpYuncken2019}, Proposition 70]\label{Proposition: Pseudos exact sequence}
    Let \(G\) be a groupoid with filtered algebroid \(\A_1\subset \A\) and corresponding osculating groupoid \(\AAA\). We get the exact sequence for all \(k \in \mathbb{Z}\):
    \[\xymatrix{0 \ar[r] & \Psi^{k-1}_H(G) \ar[r] & \Psi^k_H(G) \ar[r] & \Sigma^k(\AAA) \ar[r] & 0.}\]
\end{prop}

\begin{thm}[\cite{vanErpYuncken2019}, Theorem 59] \label{Theorem: Asymptotic completeness}
    Let \(G\) be a Lie groupoid with filtered algebroid. Let \(P_k \in \Psi^{n_k}_H(G), k\in \mathbb{N}\) be a sequence of operators, with \((n_k)_{k\in \mathbb{N}}\) decreasing to \(-\infty\). There is an operator \(P \in \Psi^{n_0}_H(G)\) such that
    \[\forall k\in \mathbb{N}, P - \sum_{j = 0}^k P_k \in \Psi^{n_{k+1}}_H(G).\]
    Such an operator is unique modulo \(\Psi^{-\infty}_H(G)\).
\end{thm}

\begin{defn}
    An operator \(P\) satisfying the condition of the previous theorem is called an asymptotic limit to the sequence of operators \((P_k)_{k\geq 0}\). We write
    \[P \sim \sum_{k\geq 0} P_k.\]
\end{defn}

\begin{cor}\label{Corollary: Total symbol}
    Let \(G\) be a Lie groupoid with filtered algebroid. There is a linear isomorphism
    \[\faktor{\Psi^0_H(G)}{\Psi^{-\infty}_H(G)} \cong \prod_{k\geq 0}\Sigma^{-k}(\AAA).\]
\end{cor}
\begin{proof}
    We fix for every \(k\) a map \(\sigma \mapsto P_{\sigma}^{(k)}\) which is a section to the symbol map \(\sigma^k_H \colon \Psi^{k}_H(G) \to \Sigma^k(\AAA)\). We can then define a map
    \[\prod_{k\geq 0}\Sigma^{-k}(\AAA) \to \faktor{\Psi^0_H(G)}{\Psi^{-\infty}_H(G)}\]
    by sending a sequence \((\sigma_k)_{k\geq 0}\) to the class of an asymptotic limit of the sequence of operators \((P_{\sigma_k}^{(k)})_{k\geq 0} \). Such an asymptotic limit is unique modulo \(\Psi^{-\infty}_H(G)\) so the map is well defined.

    Conversely, given \(P \in \Psi^0_H(G)\), we can define a sequence of symbols inductively. Define first \(\sigma_0 := \sigma^0_H(P)\) and \(P_1 = P - P^{(0)}_{\sigma_0}\). We have \(P_1 \in \Psi^{-1}_H(G)\). We can then inductively define symbols \(\sigma_k \in \Sigma^k(\AAA)\) so that
    \[\forall n \geq 0, P_n = P - \sum_{k = 0}^{n-1}P_{\sigma_k}^{(-k)} \in \Psi^{-n}_H(G),\]
    by setting \(\sigma_{n} = \sigma^{-n}_H(P_n)\). Two operators that are equal modulo \(\Psi^{-\infty}_H(G)\) would produce the same sequence of symbols so we get a map
    \[\faktor{\Psi^0_H(G)}{\Psi^{-\infty}_H(G)} \to \prod_{k\geq 0}\Sigma^{-k}(\AAA).\]
    This map is the inverse of the other one above because by construction we have
    \[P \sim \sum_{k\geq 0} P_{\sigma_k}^{(-k)}.\]
\end{proof}

This corollary gives the idea of the construction of the star-product. Using Toeplitz operators, we can reduce the symbol space from \(\Sigma^k(\AAA)\) to \(\CCC^{\infty}(M)\). We then obtain formal power series with a similar asymptotic limit argument.

\begin{rem}
    Choosing sections in the proof can be done explicitly by choosing appropriate exponential cut-offs as in \cite{vanErpYuncken2019}. In particular, we may choose the operators with support as close to the units as we want.
\end{rem}

Let \(G\) be a local Lie groupoid integrating \(\A\) and let \(\mathcal{U}\subset G^{(2)}\) be the set of composable pairs on which the multiplication is defined. We can define the set \(\Psi^{\bullet}_H(G)\) the same way as on a groupoid. It is however no longer an algebra since the convolution of kernels uses the groupoid multiplication. Notice though that if \(P,Q \in \Psi^{\bullet}_H(G)\) have respective support \(A,B\subset G\) and \(A \ _r\! \times_s B \subset \mathcal{U}\), then \(PQ\) is well defined, is a pseudodifferential operator, and has support included in \(AB:= \{m(g,h), g\in A, h\in B\}\subset G\). We call two such subsets \(A,B \subset G\) composable. We call a subset \(B \in G\) \(n\)-composable, \(n\in \mathbb{N}\) if \(B\) is \((n-1)\)-composable and \(B\) and \(B^{n-1}\) are composable (and so are \(B^{n-1}\) and \(B\)). By taking arbitrarily small neighborhood of the units in \(G\), we can find \(n\)-composable subsets that contain the units for arbitrary \(n\). We can also require that such a set (and all it's well defined power) to be included in the domain of the inverse \(\mathcal{V}\subset G\).

\begin{rem}
    Even though we can take \(n\) as large as we want, this might shrink the corresponding set in the process. If \(\A\) is not integrable, there will be no neighborhood of the units that is \(n\)-composable for every \(n \in \mathbb{N}\). If \(U \subset G\) was such a set, the set \(\bigcup_{n\in \mathbb{N}}U^n\) would be a Lie groupoid integrating \(\A\).
\end{rem}

\begin{prop}
    If \(G\) is a local groupoid, then the space \(G^{ad}\) is also endowed with a natural local groupoid strucure, which restricts on each fiber of \(G^{ad} \to \R\) to the local Lie groupoid structure of \(G\), or to the groupoid structure of \(\AAA\).
\end{prop}

\begin{proof}
    The source, range and unit maps are defined by functoriality of the deformation to the normal cone. Let \(\mathcal{U} \subset G^{(2)}\) and \(\mathcal{V} \subset G\) be the open sets on which the multiplication and inverse are defined. As a set, we have
    \[(G^{ad})^{(2)} = G^{(2)} \times \R^* \sqcup \AAA \times_M\AAA\times \{0\}.\]
    We then define
    \begin{align*}
        \mathcal{U}^{ad} & := \mathcal{U} \times \R^* \sqcup \AAA \times_M \AAA \times \{0\}, \\
        \mathcal{V}^{ad} & := \mathcal{V} \times \R^* \sqcup \AAA \times \{0\}.
    \end{align*}
    Since \(\mathcal{U}\) is open in \(G^{(2)}\) and \(\mathcal{V}\) open in \(G\) then \(\mathcal{U}^{ad}\) and \(\mathcal{V}^{ad}\) are open in \((G^{ad})^{(2)}\) and \(G^{ad}\) respectively. We can then define the multiplication and inverse the same way they would be defined for the adiabatic groupoid:
    \begin{align*} m((g,t),(g',t) )        & := (gg',t),        & \quad (g,g') \in \mathcal{U}, t \neq 0 \\
               m((x,\xi,0),(x,\xi',0)) & := (x,\xi\xi',0),  & \quad x \in M, \xi,\xi' \in \AAA_x     \\
               i(g,t)                  & := (g^{-1},t),     & \quad g\in \mathcal{V}, t\neq 0        \\
               i(x,\xi,0)              & := (x,\xi^{-1},0), & \quad  x\in M, \xi \in \AAA_x.
    \end{align*}
    These maps endow \(G^{ad}\) with a local Lie groupoid structure.
\end{proof}

We can define pseudodifferential operators on a local Lie groupoid the same way as they were defined on a Lie groupoid. To be able to compose two operators however, we need to ensure that their supports are composable. If that is the case, the convolution of the two operators is still a pseudodifferential operator as in the Lie groupoid case.

\begin{defn}
    Let \(U \subset G_c\) be an open neighborhood of the units. We write \(\Psi^k_H(U) \subset \Psi^k_H(G_c) \) for the subset of pseudodifferential operators that are supported in \(U\).
\end{defn}
If \(U,V \subset G_c\) are composable, we get a product
\[\Psi^k_H(U)\times \Psi^{\ell}_H(V) \to \Psi^{k+\ell}_H(UV).\]
In particular if \(U \subset G_c\) is 2-composable, then we get a product:
\[\Psi^k_H(U)\times \Psi^{\ell}_H(U) \to \Psi^{k+\ell}_H(U^2).\]

Since the principal symbol only depends on the behavior of the operator near the units, we get for every \(U\) a map
\(\sigma^k_H\colon \Psi^k_H(U) \to \Sigma^k(\AAA)\) such that if \(U \subset V\) then the following diagram commutes:
\[\xymatrix{\Psi^k_H(U) \ar[dr]^{\sigma^k_H} \ar@{^{(}->}[d] \\
        \Psi^k_H(V) \ar[r]^{\sigma^k_H} & \Sigma^k(\AAA) .}\]
In particular the notation \(\sigma^k_H\) is unambiguous, despite the lack of reference to the support. Moreover the principal symbol is still multiplicative in the following sense. If \(U \subset G_c\) is \(2\)-composable, then we get the commuting diagram, where the vertical arrows correspond to composition of operators and products of symbols:
\[\xymatrixcolsep{5pc}\xymatrix{\Psi^k_H(U) \times \Psi^{\ell}_H(U) \ar[r]^-{\sigma^k_H \otimes\sigma^{\ell}_H} \ar[d] & \Sigma^k(\AAA) \times \Sigma^{\ell}(\AAA)\ar[d] \\
    \Psi^{k+\ell}_H(U^2) \ar[r]^{\sigma^{k+\ell}_H} & \Sigma^{k+\ell}(\AAA).}\]

\section{Toeplitz operators}\label{Section: Toeplitz operators}

Classically, Toeplitz operators arise as the compression of multiplication operators by a projection. Typically on the circle \(\mathbb{S}^1\), we can consider the Hardy space \(H^2 \subset L^2(\mathbb{S}^1)\), the subspace of \(L^2(\mathbb{S}^1)\) of functions that extend holomorphically to the whole unit disk. Let \(S\) be the orthogonal projector onto that subspace. Toeplitz operators are operators of the form \(SM_fS\) where \(M_f\) is the multiplication operator by a continuous function \(f \in \CCC(\mathbb{S}^1)\). The same construction applies if we replace \(\mathbb{S}^1\) by the boundary of a complex manifold, under a pseudoconvexity assumption. This pseudoconvexity actually implies that the boundary is a contact manifold. The construction of the Szegö projector and the corresponding Toeplitz operators has been extended to contact manifolds \cite{BoutetdeMonvelGuillemin1981,EpsteinMelrose1998,EpsteinMelroseUnpublished} (two constructions use a priori different operators to define the Szegö projection but are equivalent). It was also recently extended to contact fibrations in \cite{CrenRezaei2026} and corresponds to families of Toeplitz operators parameterized by a base space.

Let \((\A_s,\omega)\) be a symplectic algebroid over a manifold \(M\), and \(\A_c\) its corresponding central extension. Let \(\mathfrak{A}\) be the corresponding osculating groupoid. Since we took the central extension of a symplectic algebroid, it is a bundle of Heisenberg algebras.

Heisenberg groups have two types of irreducible unitary representations: the infinite dimensional ones, parameterized by \(\R^*\) and the characters. The characters of \(\mathfrak{A}\) are parameterized by \(\A_s^*\). Its infinite dimensional representations are parameterized by \(M \times \R^*\). The representation corresponding to \(\lambda\neq 0\) is characterized by the fact that it sends \(Z\), the generator of the center, to \(\mathrm{i}\lambda\mathrm{Id}\). We will use the infinite dimensional representations to define the symbol of the Szegö projection.

Let \(J\colon \A_s \to \A_s\) be a complex structure compatible with \(\omega\). This means that we require \(\omega(J\cdot,J\cdot) = \omega\) and \(\omega(J\cdot,\cdot)\) to be a positive bilinear form. Such a complex structure always exists. We can take the complexified bundle, decomposed into eigenspaces for \(J\) \(\A_s\otimes \CC = \A_s^{1,0} \oplus \A_s^{0,1}\). The metric induced by \(\omega\) and \(J\) then provides a hermitian metric \(h\) on \(\A_s^{1,0}\). Consider the bundle of symmetric Fock spaces:
\[\F^+(\A_s^{1,0}) = \bigoplus_{n \in \mathbb{N}}^{\perp} \mathrm{Sym}^n\left(\A_s^{1,0}\right),\]
where \(\mathrm{Sym}^n\left(\A_1^{1,0}\right)\) is taken with the metric \(h^{\otimes n}\) and the sum is an \(\ell^2\)-sum. We obtain this way a Hilbert bundle.

The osculating group \(\AAA_x, x \in M\) acts naturally on the fiber of \(\F^+(\A_s^{1,0})_x\) defining an action \(\pi_+\). We can rescale this action and write \(\pi_{\lambda} := \sqrt{\lambda}\pi_+, \lambda > 0\). If we reverse the orientation, i.e. consider the opposite symplectic structure \(-\omega\), we obtain a similar action \(\pi_- \colon \AAA_x \to \mathcal{U}(\F^+(\A_1^{0,1})_x)\). Rescaling once again \(\pi_{\lambda} = \sqrt{-\lambda}\pi_{-}, \lambda <0\). We obtain this way all the infinite dimensional irreducible representations of \(\AAA\). The particular expression of the representations \(\pi_{\pm}\) will not be used here, the interested reader may find it for instance in \cite{EpsteinMelroseUnpublished}.
%Let \(N\) be the unbounded operator on \(\F^+(H^{1,0})\) whose restriction to \(\mathrm{Sym}^n\left(\A_1^{1,0}\right)\) is \(n\mathrm{Id}\). If \(\xi \in \A_{1,x}^{1,0}\) is a unit vector, it induces the creation operator
%\[\pi_{\lambda}(\xi) = \sqrt{\lambda N} \circ \mathrm{Sym}(\xi \otimes \cdot),\]
%which is an unbounded operator with domain containing the algebraic sum of the \(\mathrm{Sym}^n\left(\A_1^{1,0}\right), n \in \mathbb{N}\). Define \(\pi_{\lambda}(\overline{\xi}) := -\pi_{\lambda}(\xi)^*\).
%A quick computations shows that \([\pi_{\lambda}(\overline{\xi}), \pi_{\lambda}(\xi)] = \)

Consider the function:
\begin{align*}
    s_0 \colon \AAA^* \setminus M & \to \R                                                    \\
    (x,\xi,\eta)                  & \mapsto \exp\left(-\frac{\omega_x(J\xi,\xi)}{\eta}\right)
\end{align*}

In the formula above, we have identified the vector bundles \(\AAA\) and \(\AAA^*\) using the symplectic form, the same formula can be written directly by replacing \(\omega\) and \(J\) by their transpose.

The map \(s_0\) is smooth and homogeneous for the Heisenberg dilations, it defines a principal symbol of order zero. In light of the above decomposition of the irreducible representations we have \(\pi_-(s_0) = 0\) and \(s_0\) also vanishes on characters. On the other hand, \(\pi_+(s_0(x))\) is the projector onto the one dimensional subspace \(\mathrm{Sym}^0\left(\A_{s,x}^{1,0}\right) \subset \F^+(\A_{s,x}^{1,0})\). The symbol \(s_0\) is called a ground state projection.

\begin{rem}[On the necessity of the Heisenberg group]
    It would be tempting to replace \(\omega\) by a cocycle which might be degenerate (like the Poisson tensor \(\pi\) on the algebroid \(T^*M \to M\) of a Poisson manifold). We can still construct a central extension and Heisenberg operators on a groupoid which integrates this extension. However the formula for \(s_0\) would not be smooth on points where the form degenerates. Moreover, even if we change the formula, we won't be able to get a non-trivial projector. Indeed, if there is a point \(x \in M\) for which \(\omega_x\) is degenerate, then we have seen in proposition \ref{Corollary: Osculating central extension} that the osculating group at \(x\) is of the product of a Heisenberg group by a positive dimensional abelian group. One can then show that the only projectors in the algebra \(\Sigma^0(\AAA_x)\) (or even its natural \(C^*\)-algebraic completion) are \(0\) and \(\mathrm{Id}\), see e.g. \cite{vanErp2011Unpublished}.
\end{rem}

\begin{defn}
    Let \(G_c\) be a Lie groupoid with algebroid \(\A_c\). A \(G_c\)-Szegö operator is an operator \(S \in \Psi^0_H(G_c)\) with \(\sigma^0_H(S)\) being a ground state projector and \(S^2 - S \in \Psi^{-\infty}_H(G_c)\).
\end{defn}

\begin{rem}
    In the literature, these operators are called formal Szegö projectors. In the groupoid setting, it is very unlikely that we can find such an operator which is actually a projector. On the other hand, let \(S \in \Psi^0_H(G_c)\) be an operator such that \(\sigma^0_H(S)\) is a ground state projection, then we have \(S^2 -S \in \Psi^{-1}_H(G_c)\). Using asymptotic completeness of the calculus, we can modify \(S\) to get a \(G_c\)-Szegö operator, with the same principal symbol.
\end{rem}

\begin{lem}
    Let \(S \in \Psi^0_H(G_c)\) with \(S^2 - S \in \Psi^{-k}_H(G_c), k \geq 1\). Then there exists \(S' \in \Psi^0_H(G_c)\) with \(\sigma^0_H(S) = \sigma^0_H(S')\) and \(S'^2 -S' \in \Psi^{-2k}_H(G_c)\).
\end{lem}

\begin{proof}
    Let \(\Delta := S^2 - S \in \Psi^{-k}_H(G_c)\). At the symbolic level, \(S\) is a projector and the corresponding symmetry is (the principal symbol of) \(2S-1\). We have
    \begin{align*}
        (2S-1)^2 & = 4S^2 - 4 S +1   \\
                 & = 1 + 4 (S^2 - S) \\
                 & = 1+4\Delta.
    \end{align*}
    Let \(S' := S - (2S-1)\Delta\). Since \(\Delta \in \Psi^{-k}_H(G_c)\) and \(k > 0\) then \(S'\) and \(S\) have the same principal symbol. Moreover
    \[S'^2 = S^2 - 2S(2S-1)\Delta + (2S-1)^2\Delta^2,\]
    and so
    \begin{align*}
        S'^2-S' & = \Delta - (2S-1)^2\Delta + (2S-1)^2\Delta^2        \\
                & = \Delta - (1+4\Delta)\Delta + (1+4\Delta) \Delta^2 \\
                & = -3\Delta^2 + 4\Delta^3.
    \end{align*}
    Therefore \(S'^2-S' \in \Psi^{-2k}_H(G_c)\).
\end{proof}

\begin{prop}
    For any ground state projection \(s_0 \in \Sigma^0(\AAA)\), there exists a \(G_c\)-Szegö operator with principal symbol \(s_0\).
\end{prop}

\begin{proof}
    Fix a ground state projection \(s_0\) and take any \(S_0 \in \Psi^0_H(G_c)\) with principal symbol \(s_0\). We apply the previous lemma inductively to define:
    \[S_{k+1} = S_k - (2S_k-1)(S_k^2 -S_k).\]
    Since \(s_0\) is a projector, we have \(S_0^2-S_0 \in \Psi^{-1}_H(G_c)\). The lemma shows that \(S_k^2 - S_k \in \Psi^{-2^k}_H(G_c)\). We now have
    \[S_k = S_0 + \sum_{j = 0}^{k-1} (2S_j-1)(S_j - S_j^2)\]
    with \((2S_j-1)(S_j - S_j^2)\in \Psi^{-2^j}_H(G_c)\). By Theorem \ref{Theorem: Asymptotic completeness} we may find \(S \in \Psi^0_H(G_c)\) with
    \[S \sim S_0 + \sum_{j \geq  0} (2S_j-1)(S_j - S_j^2).\]
    By construction we have \(S^2 - S \in \Psi^{-\infty}_H(G_c)\) and \(\sigma^0_H(S) = \sigma^0_H(S_0) = s_0\). Therefore \(S\) is a \(G_c\)-Szegö operator with principal symbol \(s_0\).
\end{proof}

We now fix a \(G_c\)-Szegö operator \(S\).

\begin{defn}
    A \(G_c\)-Toeplitz operator of order \(k \in \mathbb{Z}\) is an operator of the form \(SPS\) where \(P \in \Psi^k_H(G_c)\). We denote by \(\mathcal{T}^k(G_c)\) the set of such operators plus regularizing operators, i.e.
    \[\mathcal{T}^k(G_c) := S\Psi^k_H(G_c)S + \Psi^{-\infty}_H(G_c).\]
\end{defn}

Technically the set of these operators depends on the chosen \(G_c\)-Szegö operator. We choose to omit the mention of this choice in the name and notation as it does not affect the properties of these operators that we describe below.

\begin{rem}
    Adding the regularizing operators to the set of Toeplitz operators seems a bit artificial and is only there because we consider formal Szegö operators. This assumption comes into play in the proof of Theorem \ref{Theorem: Toeplitz exact sequence} below but does not change the symbolic properties of the operators. One can get rid of this artificiality if \(S\) is an actual projector or more generally if \(S^2 = S + SRS\) with \(R \in \Psi^{-\infty}_H(G_c)\), i.e. if \(S\) is a projector modulo smoothing Toeplitz operators. With these choices we have \(\T^{-\infty}(G_c) = \Psi^{-\infty}_H(G_c) \supsetneq S\Psi^{-\infty}_H(G_c)S\).
\end{rem}

Let \(k\in \mathbb{Z}\). Since \(S\) is of order \(0\), we have \(\T^k(G_c) \subset \Psi^k_H(G_c)\). Moreover the product of two \(G_c\)-Toeplitz operators is another \(G_c\)-Toeplitz operator (of order the sum of the two previous ones).

The key feature of Toeplitz operators are their symbolic properties which we know explain. Let \(SPS \in \T^k(G_c)\) be a \(G_c\)-Toeplitz operator, \(k\in \mathbb{Z}\). By multiplicativity of the principal symbol, we have \(\sigma^k(SPS) = \sigma^0(S)\sigma^k(P)\sigma^0(S)\). From the construction of \(S\), more specifically, its symbol, we obtain directly that \(\pi_-(\sigma^k_H(SPS)) = 0\). By continuity the equatorial symbol vanishes as well so we are left with the positive part of the symbol only. On the positive representation, \(\pi_+(\sigma^0_H(S))\) is a rank one projector on each fiber. Therefore, for \(x\in M\), \(\sigma^k_{H}(SPS)(x)\) is an operator on the range of \(\sigma^0_H(S)(x)\) which is one dimensional. Moreover, the collection of range of \(\sigma^0_H(S)(x), x \in M\) forms a trivial line bundle. Consequently, there exists a unique function \(f \in \CCC^{\infty}(M)\) such that
\[\forall x\in M, \pi_+(\sigma^k_H(SPS)(x)) = f(x) \pi_+(\sigma^0_H(S)(x)).\]
This function does not depend on \(P\) but on the operator \(SPS\), as it was constructed from its principal symbol.

\begin{defn}
    The function \(f\) above is called the principal symbol of the Toeplitz operator \(SPS\). We write it as \(\sigma^k(SPS)\).
\end{defn}

\begin{ex}
    Let \(f \in \CCC^{\infty}(M)\) and write \(M_f\) for operator of multiplication by \(f\). We have \(M_f \in \Psi^0_H(G_c)\), therefore we can construct the Toeplitz operator \(T_f:= SM_fS\in \T^0(G_c)\). Its principal symbol is the function \(f\).
\end{ex}

\begin{thm}\label{Theorem: Toeplitz exact sequence}
    For \(k\in \mathbb{Z}\), we have the exact sequence
    \[\xymatrix{0 \ar[r] & \T^{k-1}(G_c) \ar[r] & \T^k(G_c) \ar[r] & \CCC^{\infty}(M) \ar[r] & 0.}\]
\end{thm}

\begin{proof}
    The sequence is obtained from Proposition \ref{Proposition: Pseudos exact sequence}. Surjectivity of the principal symbol is a consequence of the following. Let \(f \in \CCC^{\infty}(M)\) and \(P \in \Psi^k_H(G_c)\) with \(\pi_+(\sigma^k(P)) = \mathrm{Id}\). We get a Toeplitz operator \(SM_fPS \in \T^k(G_c)\) and its principal symbol is \(f\).

    The only remaining thing to show is exactness in the middle. Let \(SPS\in \T^k(G_c)\) with vanishing principal symbol. We have to show that we can write \(SPS = SQS\) with \(Q \in \Psi^{k-1}_H(G_c)\). We have \(\sigma^k_H(SPS) = 0\), therefore, \(SPS = P' \in \Psi^{k-1}_H(G_c)\). Now \(SP'S \in \T^{k-1}(G_c)\) and \(SP'S = S^2PS^2 \equiv SPS \mod \Psi^{-\infty}_H(G_c)\). Therefore \(SPS \in \T^{k-1}(G_c)\).
\end{proof}

Recall that elements of \(\Psi^{\bullet}_H(G_c)\) are particular cases of \(G_c\)-operators. As such, they can be decomposed as a smooth family of operators on the range fibers. For \(P \in \Psi^k_H(G_c)\), write \(P_x \in \Psi^k_H((G_c)^x)\) the operator on the range fiber. If \(T \in \T^k(G_c)\), then \(T_x\) is a Toeplitz operator (modulo smoothing operators) on the contact manifold \((G_c)^x\) as in \cite{Melrose2004}.

We now adapt the previous construction to local groupoids, \(G_c\) now denotes a local groupoid integrating \(\A_c\). Since \(S\) is only required to be a projection modulo smoothing operators, we can assume its support to be contained in any neighborhood of the units while still requiring it to be a \(G_c\)-Szegö operator. Let us fix a compatible complex structure \(J \colon \A_s \to \A_s\) and let \(s_0 \in \Sigma^0(\AAA)\) be the corresponding ground state projector.

\begin{prop}
    For every neighborhood of the units \(U \subset G_c\), there exists an operator \(S_U \in \Psi^0_H(U)\) which is a \(G_c\)-Szegö operator relative to the ground state projection \(s_0\), i.e. \(S_U^2 - S_U \in \Psi^{-\infty}_H(U^2)\) and \(\sigma^0(S_U) = s_0\).

    Moreover, if \(U \subset V\), and we have an operator \(S_V\) as above, we can construct an operator \(S_U\) as above and such that \(S_U - S_V \in \Psi^{-\infty}_H(V)\).
\end{prop}

\begin{defn}
    Let \(U\subset G_c\) be a \(3\)-composable subset, \(S_U\) a \(G_c\)-Szegö operator supported in \(U\). We define Toeplitz operators by
    \[\T^k(U) := S_U \Psi^k_H(U)S_U + \Psi^{\-\infty}_H(U^3).\]
\end{defn}

\begin{rem}
    Despite the notation, elements in \(\T^k(U)\) have a support included in \(U^3\) and not in \(U\).
\end{rem}

The exact sequence of Proposition \ref{Proposition: Pseudos exact sequence} doesn't à priori hold with these extra support conditions. Notice indeed that in the proof, to produce a Toeplitz operator of order one less, we had to multiply the initial operator by the \(G_c\)-Szegö. This would only show that if \(SPS \in \T^k(U)\) has vanishing principal symbol, then \(S_UPS_U \in \T^{k-1}(U^3)\). Despite that slight complication, we still have a surjective map
\[\faktor{\T^k(U)}{\T^{k-1}(U)} \twoheadrightarrow \CCC^{\infty}(M).\]

Moreover, if \(U \subset V\) and we take \(S_U\) such that \(S_U - S_V \in \Psi^{-\infty}_H(V)\), then the following diagram commutes
\[\xymatrix{\faktor{\T^k(U)}{\T^{k-1}(U)} \ar[d] \ar[rd] \\
        \faktor{\T^k(V)}{\T^{k-1}(V)} \ar[r] & \CCC^{\infty}(M).}\]
The vertical arrow is induced by the inclusion \(\T^k(U) \to \T^k(V)\) but might fail to be injective in general.

\begin{thm}\label{Theorem: principal symbol local Toeplitz}
    Let \((U_n)_{n\in \mathbb{N}}\) be a decreasing sequence of open neighborhood of the units in \(G_c\) such that \(\bigcap_{n\in \mathbb{N}} U_n = M\). Consider Szegö projectors \(S_n \in \Psi^0_H(U_n)\) such that if \(n \leq m\) then \(S_n - S_m \in \Psi^{-\infty}_H(U_m)\). Then for all \(k \in \mathbb{Z}\), there is an isomorphism
    \[\varprojlim_{n\in\mathbb{N}}\faktor{\T^k(U_n)}{\T^{k-1}(U_n)} \cong \CCC^{\infty}(M).\]

\end{thm}

\begin{proof}
    The maps \(\faktor{\T^k(U_n)}{\T^{k-1}(U_n)} \to \CCC^{\infty}(M)\) commute with the inverse system \(\left(\faktor{\T^k(U_n)}{\T^{k-1}(U_n)}\right)_{n\in\mathbb{N}}\). We thus have a map
    \[\sigma\colon \varprojlim_{n\in\mathbb{N}}\faktor{\T^k(U_n)}{\T^{k-1}(U_n)} \to \CCC^{\infty}(M). \]
    The inverse of this map can be constructed as follows. Let \(\widetilde{D}_n \in \Psi^k_H(U_n)\) be an operator such that the principal symbol of the Toeplitz operator \(S_n\widetilde{D}_nS_n\) is \(x \mapsto 1\). We choose them so that if \(n \geq m\) then \(\widetilde{D}_{n} - \widetilde{D}_{m} \in \Psi^{k-1}_H(U_m)\). For \(f \in \CCC^{\infty}(M)\) define \(T_f^{k,(n)} := S_n f\widetilde{D}_n S_n\). We obtain a family of linear maps
    \begin{align*}
        T^{k,(n)} \colon \CCC^{\infty}(M) & \to \T^0(U_n)        \\
        f                                 & \mapsto T_f^{k,(n)}.
    \end{align*}
    These maps are not compatible with the inverse system \((\T^k(U_n))_{n\in \mathbb{N}}\). However we have for \(n \geq m\)
    \[T^{k,(n)}_f \equiv T^{k,(m)}_f \mod \T^{k-1}(U_m).\]
    Therefore the maps \(f \mapsto T^{k,(n)}_f \mod \T^{k-1}(U_n)\) are compatible with the inverse system \(\left(\faktor{\T^k(U_n)}{\T^{k-1}(U_n)}\right)_{n\in\mathbb{N}}\). This defines a map
    \begin{align*}
        \overline{T} \colon \CCC^{\infty}(M) & \longrightarrow \varprojlim_{n\in\mathbb{N}}\faktor{\T^k(U_n)}{\T^{k-1}(U_n)} \\
        f                                    & \longmapsto \left(T^{k,(n)}_f \mod \T^{k-1}(U_n) \right)_{n\in\mathbb{N}}.
    \end{align*}
    The principal symbol of \(T^{k,(n)}_f \) is \(f\), this gives \(\sigma\circ \overline{T} = \mathrm{Id}\). On the other hand, assume that \(\sigma((S_nP_nS_n)_{n\in \mathbb{N}}) = 0\). For \(n\) big enough, we may assume that \(U_n\) is \(5\)-composable so that
    \[S_nP_nS_n = S_n(S_nP_nS_n)S_n \mod \Psi^{-\infty}_H(U_n^5).\]
    But \(\sigma^k(S_nP_nS_n) = 0\) so \(S_nP_nS_n \in \Psi^{k-1}_H(U_n^3)\). Let \(m \in \mathbb{N}\), there exists \(n \geq m\) such that \(U_n^3 \subset U_m\). We then have \(S_nP_nS_n \in \T^{k-1}(U_m)\) and \((S_nP_nS_n)_{n\geq 0} = 0\) in \(\varprojlim_{n\in\mathbb{N}}\faktor{\T^k(U_n)}{\T^{k-1}(U_n)}\). This shows that \(\sigma\) is injective. It is thus an isomorphism with inverse \(\overline{T}\).
\end{proof}

\section{Construction of the star-product}\label{Section: Star product}

In this section we construct the star-product using Toeplitz operators. Before that, we begin by exploring the relationship between the Poisson brackets $\{\cdot,\cdot\}_c$ and $\{\cdot,\cdot\}$ of the manifold  $A_c^*$ as the dual of a Lie algebroid, and the manifold $M$ with its Poisson structure induced by $A_s$ respectively.
Remember that the latter is given by the commutative diagram:
\[\xymatrix{\A_s \ar[r]^{\rho} & TM \\
	\A_s^* \ar[u]^{\omega^{\#}} & T^*M \ar[l]^{\rho^*} \ar[u]_{\pi^{\#}}}\]
that is, 
\begin{equation*}
	\forall f,g\in \CCC^\infty(M),\quad \{f,g\} = \langle df , \pi^\# dg\rangle = \langle \rho_s^*df , (\rho_s^*dg)^\#\rangle = \omega( (\rho_s^*df)^\# , (\rho_s^*dg)^\#).
\end{equation*}
Therefore, denoting $X_a = (\rho_s^*da)^\#$ the Hamiltonian field of $a\in\CCC^\infty(M)$: 
\[ \{f,g\} = \omega(X_f,X_g) = -\rho_s(X_f)g = \rho_s(X_g)f.\]
Let us note $\widetilde{f}$ the pull-back to $A_c^*$ of $f\in\CCC^\infty(M)$ and $\ell_X : \Gamma (A_c^*)\to \CCC^\infty(M), \xi\mapsto \langle \xi, X\rangle$ for any $X\in \Gamma(A_c)$. 
Since the Poisson structure of $A_c^*$ is linear \cite[Section 10.3]{Mackenzie2005}, we have: 
\begin{align*}
\forall f,g\in\CCC^\infty(M),\quad  \{\ell_{(X_f,0)}, \ell_{(X_g,0)}\}_c & = \ell_{[(X_f,0),(X_g,0)]} = \ell_{([X_f,X_g],\omega(X_f,X_g))}\\
& =\ell_{([X_f,X_g],0))}+ \widetilde{\omega(X_f,X_g)}\ell_{(0,1)} \\
& = \ell_{([X_f,X_g],0))} + \eta \widetilde{\omega(X_f,X_g)}
\end{align*}
Let us consider \( \cS = \{ (x,0,\eta)\in A_c^*\ \mid \ \eta > 0 \} \). The previous computations show that:
\[
\forall f,g\in\CCC^\infty(M),\quad  \widetilde{\{f,g\}} = \frac{1}{\eta}\{\ell_{(X_f,0)}, \ell_{(X_g,0)}\}_c \text{ on } \mathcal{S}.
\]

The manifold \(\cS\) is a Dirac submanifold in the sense of \cite{Xu2003}. It is thus endowed with a Poisson bracket \(\{\cdot,\cdot\}_{\cS}\) of it's own, which resembles the one of \(\A_c^*\) but with a correction term. It is obtained similarly as the one on a symplectic submanifold of a symplectic manifold (see for instance \cite[Section 8.5]{MarsdenRatiu1998}). We relate it to the Poisson bracket of $M$. 

\begin{lem}\label{Lemma: Poisson brackets comparison}
For \(f,g\in \CCC^{\infty}(M)\), we have 
\[\{\widetilde{f},\widetilde{g}\}_{\cS} = \frac{1}{\eta}\widetilde{\{f,g\}}.\]
\end{lem}

\begin{proof}
Let $(X_i)_{1\leq i \leq 2m}$ be a local basis of sections of $A_s$ over an open subset $U\subset M$ and $\xi = (\xi_i)_{1\leq i \leq 2m}$ the dual coordinates in the fibers of $A_s^*$. We see the \(\xi_i\) as functions on \(\A_c^*\) through the inclution \(\A_s^*\subset \A_c^*\). We have : 
\[ \{ \xi_i,\xi_j\}_c = \ell_{[(X_i,0),(X_j,0)]} = \ell_{([X_i,X_j],0)}+ \eta \widetilde{\omega(X_i,X_j)}. \]
Set $C(x,\xi,\eta) =  (\{ \xi_i,\xi_j\}_c)$ and  $\Omega = (\widetilde{\omega(X_i,X_j)})$.  In restriction to $ \cS $, we have $C = \eta \Omega$. Thus $C$ is invertible in a (conic) neighborhood of  $\cS$ into $A_c^*$. 
Now for any functions $f,g\in \CCC^\infty(\cS)$, the function on $\pi^{-1}(U)\cap \cS$ defined by:
\[  \{f,g\}_\cS =  \{F,G\}_c - \sum_{i,j} [C^{-1}]_{i,j}\{F,\xi_i\}_c\{\xi_j,G\}_c \]
does not depend on the choice of $F$ and $G$ extending smoothly $f$ and $g$ to $A_c^*$, nor the choice of the local basis, and gives the Poisson structure of $\cS$. Let $f,g\in \CCC(M)$ and denote by $\widetilde{f},\widetilde{g}$ both their pull-back to $\cS$ and $A_c^*$. By the properties of the Linear Poisson structure of $A_c^*$ we have 
 \[\{\xi_i, \widetilde{f}\}_c =  \{\ell_{(X_i,0)}, \widetilde{f}\}_c = \widetilde{\rho(X_i)f}\]
 and 
 \[ \{\widetilde{f},\widetilde{g}\}_c = 0.\]
 Therefore
\begin{align*}
	\{\widetilde{f},\widetilde{g}\}_\cS & = \frac{1}{\eta}\sum_{i,j}[\Omega^{-1}]_{i,j}\widetilde{\rho(X_i)f}\widetilde{\rho(X_j)g} \\
	& =  \frac{1}{\eta}\sum_{i} \widetilde{\rho^*df(X_i)}(\sum_j [\Omega^{-1}]_{i,j}\widetilde{\rho^*dg(X_j)}) \\
	& =  \frac{1}{\eta}\sum_{i} \widetilde{(\rho^*df)^i (\rho^*dg)^\#_i }= \frac{1}{\eta}\widetilde{\langle \rho^*df, X_g\rangle}.\\
\end{align*}
That is:
\[  \{f,g\}_\cS= \frac{1}{\eta} \widetilde{\{f,g\}}.\]
\end{proof}

We now give the construction of the star product, starting with the integrable case. Let \(\A_s\) be a symplectic algebroid, \(\A_c\) its corresponding central extension. Let \(G_c \twoheadrightarrow M\) be a Lie groupoid integrating \(\A_c\). We fix a \(G_c\)-Szegö operator \(S \in \Psi^0_H(G_c)\).

\begin{thm}\label{Theorem: Symbol commutator}
    Let \(f,g \in \CCC^{\infty}(M)\), then \([T_f,T_g]\in \T^{-1}(G_c)\) and
    \[\sigma^{-1}([T_f,T_g]) = \frac{1}{\mathrm{i}}\{f,g\}.\]
    Here \(\{\cdot,\cdot\}\) denotes the Poisson bracket on \(M\) induced by the symplectic algebroid \(\A_s\).
\end{thm}

\begin{proof}
    Let \(s_0 = \sigma(S)\) be the ground state projection corresponding to the \(G_c\)-Szegö operator. We have for \(x\in M\),
    \begin{align*}
        \sigma^0(T_fT_g)(x) & = s_0(x) f(x) s_0^2(x) g(x)s_0(x) \\
                            & = f(x)g(x) s_0^4(x)               \\
                            & = s_0(x) f(x)g(x) s_0(x)          \\
                            & = \sigma^0(T_{fg}).
    \end{align*}
    Therefore by symmetry, \(\sigma^0([T_f,T_g]) = 0\) and \([T_f,T_g]\) has order \(-1\). We see symbols of order \(k \in \mathbb{Z}\) as smooth functions on \(\cS\) that are homogeneous of order 		\(k\).

    Computations in \cite[Proposition 11.9]{BoutetdeMonvelGuillemin1981} and \cite[Equation 283]{EpsteinMelroseUnpublished} show that we have
    \[\sigma^{-1}([T_f,T_g]) = \frac{1}{\mathrm{i}}\{\widetilde{f},\widetilde{g}\}_{\cS}.\]
    
    Now using Lemma \ref{Lemma: Poisson brackets comparison}, we obtain
    \[\sigma^{-1}([T_f,T_g]) = \frac{1}{i\eta}\widetilde{\{f,g\}}.\]
    This is indeed a symbol of order \(-1\), corresponding to the function \(\frac{1}{\mathrm{i}}\{f,g\} \in \CCC^{\infty}(M)\).
\end{proof}

\begin{thm}\label{Theorem: Star-product integrable case}
    Let \(\A_s\) be a symplectic algebroid, \(\A_c\) its corresponding central extension. Let \(G_c \twoheadrightarrow M\) be a Lie groupoid integrating \(\A_c\) and \(S\in \Psi^0_H(G_c)\) a \(G_c\)-Szegö operator. There is an isomorphism
    \[\faktor{\T^0(G_c)}{\T^{-\infty}(G_c)} \cong \CCC^{\infty}(M)[[\hbar]].\]
    The product of Toeplitz operators induces a star-product on \(M\) for the Poisson structure induced by \(\A_s\).
\end{thm}

\begin{proof}
    Let \(D\in \T^1(G_c)\) with principal Toeplitz symbol \(x\mapsto 1\). We will write \(D^{-1}\in \T^{-1}(G_c)\) for a parametrix of \(D\), that is \(DD^{-1} - \mathrm{Id}, D^{-1}D -\mathrm{Id} \in \T^{-\infty}(G_c)\).

    Let \(P \in \T^0(G_c)\) and denote by \(f_0\) its principal Toeplitz symbol. We have \(\sigma^0(P - T_{f_0}) = 0\) so \(R_1 := T_{f_0} - P \in \T^{-1}(G_c)\). Let \(f_1 = \sigma^0(DR_1)\), we have
    \[P = T_{f_0} + D^{-1} T_{f_1} + R_2, \quad R_2 \in \T^{-2}(G_c).\]
    We may proceed inductively to get functions \(f_k, k\geq 0\) and operators \(R_k \in \T^{-k}(G_c), k \geq 1\) so that
    \[\forall n \in \mathbb{N}, P = \sum_{k = 0}^{n-1}D^{-k} T_{f_k} + R_n.\]
    These functions do not depend on the class of \(P\) modulo \(\T^{-\infty}(G_c)\). We have thus defined a linear map:
    \[\faktor{\T^{0}(G_c)}{\T^{-\infty}(G_c)} \to \CCC^{\infty}(M)[[\hbar]].\]
    The inverse of this map is defined as follows. Consider a sequence of functions \((f_k)_{k\geq 0}\in \CCC^{\infty}(M)^{\mathbb{N}}\). The sequence of operators \(D^{-k}T_{f_k}, k \geq 0 \) satisfies the conditions of Theorem \ref{Theorem: Asymptotic completeness} so we may find an operator \(P \in \T^0(G_c)\) with
    \[\forall n \in \mathbb{N}, P - \sum_{k = 0}^{n-1}D^{-k} T_{f_k} \in \T^{-n}(G_c).\]
    This operator is unique modulo \(\T^{-\infty}(G_c)\) so we get a linear map
    \[\CCC^{\infty}(M)[[\hbar]] \to \faktor{\T^{0}(G_c)}{\T^{-\infty}(G_c)}.\]
    This map is the inverse to the previous one.

    Now the product of Toeplitz operators induces a product \(\ast\) on the space of formal power series \(\CCC^{\infty}(M)[[\hbar]]\). This product is associative. If \(f,g \in \CCC^{\infty}(M)\) are seen as elements of order \(0\), we have
    \begin{align*}
        f\ast g & = T_fT_g\mod \T^{-\infty}(G_c)  \\
                & = T_{fg} \mod \T^{-\infty}(G_c) \\
                & = fg + O(\hbar).
    \end{align*}
    Similarly, using Theorem \ref{Theorem: Symbol commutator} we get
    \begin{align*}
        f\ast g - g \ast f & = [T_f,T_g] \mod \T^{-\infty}(G_c)          \\
                           & = -\mathrm{i} D^{-1} T_{\{f,g\}} \mod \T^{-2}(G_c) \\
                           & = -\mathfrak{i}\hbar \{f,g\} + O(\hbar^2).
    \end{align*}
    Therefore we obtain a product on \(\CCC^{\infty}(M)[[t]]\) which is a star-product for \((M,\pi)\) by taking \(t = -i\hbar\). 
\end{proof}

We now construct the star-product in the general case. The idea is similar to the integrable case but a projective limit argument is needed to overcome the composition problem, in the spirit of Theorem \ref{Theorem: principal symbol local Toeplitz}.

We now take \(G_c\) to be a local Lie groupoid integrating \(\A_c\). Let \((U_n)_{n\in \mathbb{N}}\) be a decreasing sequence of open neighborhood of the units in \(G_c\) such that \(\bigcap_{n\in \mathbb{N}} U_n = M\). Consider Szegö projectors \(S_n \in \Psi^0_H(U_n)\) such that if \(n \leq m\) then \(S_n - S_m \in \Psi^{-\infty}_H(U_m)\).
We also consider operators \(D_n \in \T^1(U_n)\) and parametrices \(D_n^{-1} \in \T^{-1}(U_n)\). This means that \(D_nD_n^{-1}-\mathrm{Id},D_n^{-1}D_n-\mathrm{Id} \in \Psi^{-\infty}(U_n^6)\) (this makes sense for \(n\) big enough). As for the operators \(S_n\), we choose the operators \(D_n\) and \(D_n^{-1}\) so that if \(m > n\) then \(D_n - D_m, D_n^{-1}-D_m^{-1} \in \Psi^{-\infty}_H(U_n^3)\).

\begin{thm}
    There is an isomorphism
    \[\varprojlim_{n\to \infty}\faktor{\T^0(U_n)}{\T^{-\infty}(U_n)} \cong \CCC^{\infty}(M)[[\hbar]].\]
    The product of Toeplitz operators induces a star-product on \(M\) for the Poisson structure induced by \(\A_s\).
\end{thm}

\begin{proof}
    Let \(\widetilde{D}^k_{n} \in \Psi^k_H(U_n)\) be an operator such that the principal symbol of the Toeplitz operator \(S_n\widetilde{D}^k_nS_n\) is \(x \mapsto 1\). We choose them so that if \(n \geq m\) then \(\widetilde{D}^k_{n} - \widetilde{D}^k_{m} \in \Psi^{-\infty}_H(U_m)\). For \(f \in \CCC^{\infty}(M)\) define \(T_f^{k,(n)} := S_n f\widetilde{D}^k_n S_n\). Consider an element \(\sum_{k\geq 0} f_k \hbar^k\in \CCC^{\infty}(M)[[\hbar]]\). The sequence of operators \((T_{f_k}^{-k,(n)})_{k\geq 0}\) satisfies the conditions of Theorem \ref{Theorem: Asymptotic completeness} so we may find an operator \(P_n \in \Psi^0_H(U_n^3)\) such that
    \[P_n \sim \sum_{k\geq 0} T_{f_k}^{-k,(n)}.\]
    Replacing \(P_n\) by \(S_nP_nS_n\) (this makes sense for \(n\) big enough) we still have
    \[S_nP_nS_n \sim \sum_{k\geq 0} T_{f_k}^{-k,(n)}\]
    and now \(S_mP_mS_m\in \T^0(U_n)\) for \(m\geq n\) such that \(U_m^3 \subset U_n\). The operator \(P_n\) is uniquely defined modulo regularising operators. This defines a linear map
    \[\overline{T}\colon \CCC^{\infty}(M)[[\hbar]] \to \varprojlim_{n\to \infty}\faktor{\T^0(U_n)}{\T^{-\infty}(U_n)}.\]
    The inverse of that map is obtained as follows. Consider an element \(P_n \in \T^0(U_n)\). We have \(\sigma^0(P_n) =: f_0\) for some smooth function \(f_0\in \CCC^{\infty}(M)\). The operator \(P_n - T_f^{0,(n)}\) belongs to \(\Psi^{-1}_H(U_n)\). We can also see it as en element in \(S_n\Psi^{-1}_H(U_n^3)S_n + \Psi^{-\infty}_H(U_n^5)\) if \(n\) is big enough but we may compute its Toeplitz symbol regardless, by looking at \(\sigma^{-1}_H(P_n - T_f^{0,(n)})\). We obtain another smooth function \(f_1\). We define iteratively:
    \[f_{k+1} = \sigma^{-k-1}\left(P_n - \sum_{j = 0}^k T_f^{-j,(n)}\right).\]
    If \(m\geq n\) and \(P_m - P_n \in \Psi^{-\infty}(U_n)\) then the functions \(f_k, k\geq 0\) defined from \(P_n\) and \(P_m\) are the same.
    This defines a linear map:
    \[\overline{\sigma} \colon \varprojlim_{n\to \infty}\faktor{\T^0(U_n)}{\T^{-\infty}(U_n)} \to \CCC^{\infty}(M)[[\hbar]].\]
    By construction we have
    \[P_n \sim \sum_{k\geq 0} T_{f_k}^{-k,(n)}.\]
    This relation shows that \(\overline{\sigma}\circ \overline{T}\) and that \(\overline{\sigma}\) is injective. Indeed, two operators having the same total symbol (the sequence of functions \(f_k, k \geq 0\)) would differ by a regularizing operator. The maps \(\overline{\sigma}\) and \(\overline{T}\) are then inverse to one another. The product of Toeplitz operators then induces an algebra structure on \(\CCC^{\infty}(M)[[\hbar]]\), it is a star-product by the same computations as in \ref{Theorem: Star-product integrable case}.
\end{proof}

\begin{rem}
    The projective limit argument can be replaced as follows. Consider a \(G_c\)-Szegö operator \(S_0\). Consider then \(\mathscr{S} = \{S \in \Psi^0_H(G_c), S-S_0 \in \Psi^{-\infty}_H(G_c)\}\). All the elements of \(\mathscr{S}\) are \(G_c\)-Szegö operators with the same principal symbol as \(S_0\). We redefine the set of Toeplitz operators to be
    \[\widetilde{\T}^k(G_c) = \{SPS, S\in \mathscr{S}, P\in \Psi^k_H(G_c)\}+\Psi^{-\infty}_H(G_c).\]
    Elements in \(\faktor{\widetilde{\T}^0(G_c)}{\widetilde{\T}^{-\infty}(G_c)}\) admit representatives of the form \(SPS\) with \(S\) and \(P\) having support in a neighborhood of the units as small as possible (we at least need it for the product to make sense). The same arguments as before show that
    \[\faktor{\widetilde{\T}^0(G_c)}{\widetilde{\T}^{-\infty}(G_c)} \cong \CCC^{\infty}(M)[[\hbar]].\]

    By taking representatives with appropriate support, we can always compose elements of the left hand side. The product induced on the formal power series is again a star-product by proposition \ref{Theorem: Symbol commutator}.
\end{rem}

\bibliographystyle{plain}
\bibliography{Biblio.bib}

\end{document}